\documentclass[10pt]{amsart}

\usepackage{amsmath}
\usepackage{amssymb}
\usepackage{graphicx}

\newtheorem{theorem}{Theorem}[section]
\newtheorem{corollary}[theorem]{Corollary}
\newtheorem{lemma}[theorem]{Lemma}
\newtheorem{proposition}[theorem]{Proposition}
\theoremstyle{definition}
\newtheorem{definition}[theorem]{Definition}
\newtheorem{example}[theorem]{Example}
\newtheorem{remark}[theorem]{Remark}

\numberwithin{equation}{section}

\newcommand{\B}{\mathbb B}
\newcommand{\R}{\mathbb R}
\newcommand{\N}{\mathbb N}

\newcommand{\X}{\mathbb X}

\newcommand{\nullv}{{\bf 0}}
\newcommand{\inte}{{\rm int}\,}

\newcommand{\dom}{{\rm dom}\,}
\newcommand{\grph}{{\rm grph}\,}
\newcommand{\fsubdif}{\widehat{\partial}}
\newcommand{\Coder}{\widehat{\rm D}^*}
\newcommand{\Normal}{\widehat{\rm N}}

\newcommand{\stslp}[1]{|\nabla_p #1|}
\newcommand{\sostslp}[1]{\overline{|\nabla_p #1|}{}^>}
\newcommand{\dispo}[1]{|[#1]|}

\newcommand{\ball}[2]{{\rm B}(#1, #2)}

\newcommand{\dist}[2]{{\rm dist}\left(#1,#2\right)}

\newcommand{\usreg}[2]{{\rm u.hreg}(#1, #2)}
\newcommand{\uLlsc}[2]{{\rm u.liplsc}(#1, #2)}

\newcommand{\disp}[3]{|[#1]|(#2,#3)}

\newcommand{\pcalm}[3]{{\rm clm}(#1,#2,#3)}

\title[An implicit multifunction theorem for uniform hemiregularity]
{An implicit multifunction theorem for the hemiregularity
of mappings with application to constrained optimization}

\author[A. Uderzo]{A. Uderzo}

\address[A. Uderzo]{Dept. of Mathematics and Applications,
University of Milano-Bicocca}
\email{{\tt amos.uderzo@unimib.it}}



\begin{document}

\begin{abstract}
The present paper contains some investigations about
a uniform variant of the notion of metric hemiregularity,
the latter being
a less explored property obtained by weakening metric regularity.
The introduction of such a quantitative stability property
for set-valued mappings is motivated by applications to the
penalization of constrained optimization problems, through
the notion of problem calmness. As a main result, an implicit
multifunction theorem for parameterized inclusion problems
is established, which measures the uniform hemiregularity
of the related solution mapping in terms of problem data.
A consequence on the exactness of penalty functions is discussed.
\end{abstract}

\maketitle

\begin{flushright}
to S.M. Robinson, on the occasion of his 75th birthday
\end{flushright}


\section{Introduction}    \label{Sect:1}

The key idea ispiring the Lagrangian approach to constrained
optimization is to avoid to determine all elements in the feasible
region of a given problem: in fact, solving explicitly a nonlinear equation
system is typically a task too hard to be undertaken. Sometimes,
it is even superfluous to do it, as far as local optimality is concerned.
Instead, the approach proposes to formulate optimality
conditions by filling the lack of information about the feasible
region with the usage of implicit function theorems. Thus,
after such an approach, implicit function theorems became a
crucial tool for the constraint system analysis. Historically,
constrained optimization acted as a driving force for the development
of theorems of this kind.
For instance, the celebrated Lyusternik's theorem, one of the earliest
implicit function theorems formulated in abstract spaces, which had a
remarkable impact on modern variational analysis, was established
exactly with this aim (see \cite{Lyus34}).
That said, it is not surprising that the evolution of optimization
conditioned the investigations about implicit function theorems.
Essentially, two main facts contributed to shape the evolution
process of optimization, stimulated by theoretical and applicational
reasons: an increasing complexity of constraint systems and the
appearance of nonsmoothness in problem data. Their effect, both in
formalizing and solving the resulting optimization problems, was
that equations were replaced by more general relations called
generalized equations, where set-valued mappings played a fundamental
role. In order to devise extensions of the Lagrangian approach
suitable to the new context, implicit function theorems had to be
adequated. Such a direction of research was soon clearly understood,
among the others, by S.M. Robinson, who introduced the term
``generalized equation" and provided seminal contributions to
the theory coming up around this issue (see \cite{DonRoc14,Robi79,Robi80}).
In the large variety of forms taken by the new generation of implicit
function theorems that arose with the help of techniques from
variational analysis, some common elements can be still recognized:
instead of classical functions, they speak of multifunctions,
which emerge as a solution mapping of a parameterized generalized
equation; instead of differentiability, they establish some kind
of Lipschitzian behaviour of the implicitly defined multifunctions,
along with related quantitative estimates. Both these features
seem to be rather natural within the new context. In
particular, notice that differentiability of a mapping can
be viewed as a local calmness property of the error resulting
from affine approximation of it. Moreover, what is important,
they allow to treat effectively a broad spectrum of constraint
systems.
In the impossibility of providing a comprehensive updated
account of all relevant achievements about this theme, the reader
is referred to \cite{BorZhu05,DonRoc14,Mord06,RocWet98,Schi07}
and the bibliographies therein.

The investigations exposed in the present paper proceed
along the aforementioned direction of research. In particular,
they focus on a property of uniform metric hemiregularity for the
solution mapping associated with a parameterized generalized
equation, whose interest is motivated by applications to penalty
methods in constrained optimization. This property for set-valued
mappings can be obtained as a weak variant of the more
studied and widely employed property known as metric regularity,
which describes a local Lipschitzian behaviour of multifunctions.
Even though it made its first appearance in its inverse formulation
as Lipschitz lower semicontinuity already in \cite{KlaKum02}, only
recently was explicitly formulated and investigated under
different names
\footnote{To avoid confusion with another property having
the same name (see \cite[Definition 10.6.1 (b)]{Schi07}),
instead of ``semiregularity",
which was used in \cite{Krug09,KruTha15}, in the present
paper the term ``hemiregularity", borrowed from \cite{AraMor11},
is adopted.}
(see \cite{AraMor11,DonRoc14,Krug09,KruTha15}).

The contents of the paper are arranged as follows.
In Section \ref{Sect:2}, the basic definitions are introduced,
several equivalent reformulations of uniform hemiregularity are
provided, along with some examples of uniform  hemiregular mappings.
This multiple description should help to catch connections with
similar properties and then to better understand the main phenomenon
under consideration.
In Section \ref{Sect:3} a motivation for introducing uniform
hemiregularity, coming from constrained optimization, is discussed
in detail.
Section \ref{Sect:4} contains the main result of the paper, that is
an implicit multifunction theorem. It provides a sufficient condition
for the solution mapping, associated with a parameterized inclusion
problem, to be uniformly hemiregular at a given point of its graph,
along with an estimate of the uniform hemiregularity modulus of it.
Such a result is established in a purely metric setting, by
means of a variational technique largely employed in this field
(see, for instance, \cite{BorZhu05}). Its impact on constrained
optimization in terms of conditions for the exactness of penalty
functions and relationships with the existent literature on
the subject is then discussed. A specialization of the main result
to the Asplund space setting, involving Fr\'echet coderivatives,
is also presented.

Throughout the paper the use of the basic notations is standard.
Whenever $(P,d)$ denotes a metric space, given $\bar p\in P$ and
$r\ge 0$, $\ball{\bar p}{r}=\{p\in P:\ d(p,\bar p)\le r\}$ indicates
the closed ball centred at $\bar p$ with radius $r$. In the same
setting, if $S\subseteq P$, $\dist{\bar p}{S}=\inf_{p\in S}d(\bar p,p)$ stands
for the distance of $\bar p$ from $S$, with the convention that
$\dist{\bar p}{\varnothing}=+\infty$. By $\ball{S}{r}=\{p\in P:\
\dist{p}{S}\le r\}$ the $r$-enlargement of $S$ is denoted.
By $\inte S$ the topological interior of $S$ is denoted.
Whenever $\Theta:P\rightrightarrows X$ is a set-valued mapping,
$\grph\Theta$ and $\dom\Theta$ denote the graph and the domain
of $\Theta$, respectively.
Unless otherwise indicated, all set-valued mappings will be
assumed to take closed values.
Throughout the text, the acronyms l.s.c. and u.s.c. stand for
lower semicontinuous and upper semicontinuous, respectively.
Further special notations will be introduced contextually to
their use.

\vskip1cm


\section{Uniform hemiregularity and related notions}
\label{Sect:2}

The main property under study is introduced in the following
definition.

\begin{definition}      \label{def:usreg}
Let $\Theta:P\rightrightarrows X$ be a set-valued mapping between
metric spaces and let $(\bar p,\bar x)\in\grph\Theta$. $\Theta$
is called:
\begin{itemize}

\item[(i)] ({\it metrically}) {\it hemiregular at $(\bar p,\bar x)$} if
there exist positive constants $\kappa$ and $r$ such that
$$
    \dist{\bar p}{\Theta^{-1}(x)}\le\kappa d(x,\bar x),\quad\forall
   x\in\ball{\bar x}{r};
$$

\item[(ii)] {\it uniformly} ({\it metrically}) {\it hemiregular
at $(\bar p,\bar x)$} if there
exist positive constants $\kappa$ and $r$ such that
\begin{eqnarray}   \label{in:usregdef}
    \dist{\bar p}{\Theta^{-1}(x)}\le\kappa d(x,z),\quad\forall
   x\in\ball{z}{r},\ \forall z\in\Theta(\bar p)\cap
   \ball{\bar x}{r}.
\end{eqnarray}
The value
$$
    \usreg{\Theta}{(\bar p,\bar x)}=\inf\{\kappa>0:\ \exists r>0 \hbox{ for which
    $(\ref{in:usregdef})$ holds } \}
$$
is called the {\it modulus of uniform} ({\it metric}) {\it hemiregularity}
of $\Theta$ at $(\bar p,\bar x)$.
\end{itemize}
\end{definition}

Roughly speaking, the above introduced properties refer to a
kind of ``quantitative solvability" of the systems
$$
  x\in\Theta(p),
$$
where $x$ is a parameter varying near the reference value $\bar x$
and $\bar p$ is a solution of the system $\bar x\in\Theta(p)$.
Notice that, according to the convention made about the value
of $\dist{\bar p}{\varnothing}$, if $\Theta$ is hemiregular at $(\bar p,\bar x)$,
then each of the perturbed systems must be solvable. Moreover, the
distance of the given solution $\bar p$ from the varying solution
sets must be linearly controlled by the distance of $x$ from
$\bar x$.

\begin{remark}     \label{rem:equivrefuhemic}
The property in Definition \ref{def:usreg}(ii) is clearly a
stronger variant than mere hemiregularity, even if the latter
takes place at each pair $(\bar p,z)$, with $z\in\Theta(\bar p)\cap
\ball{\bar x}{r}$. Indeed,
the constants $\kappa$ and $r$ in Definition \ref{def:usreg}(ii) are
postulated to be the same for every $z\in \Theta(\bar p)\cap
\ball{\bar x}{r}$, whence the term of the resulting property.
This uniformity requirement
enables one to reformulate such a property in a slightly different way,
that will be useful for the purposes of the present analysis:
$\Theta$ is uniformly hemiregular at $(\bar p,\bar x)$ iff there exist
positive $\kappa$ and $\delta$ such that
\begin{eqnarray}   \label{in:altusregdef}
    \dist{\bar p}{\Theta^{-1}(x)}\le\kappa\dist{x}{\Theta(\bar p)},
    \quad\forall x\in\ball{\bar x}{\delta}.
\end{eqnarray}
Indeed, if inequality $(\ref{in:usregdef})$ holds true, then for every
$x\in\ball{\bar x}{r/2}\backslash\Theta(\bar p)$ and $\epsilon
\in (0,1)$ it is possible to claim the existence of a proper $z_\epsilon\in
\Theta(\bar p)$, such that $d(x,z_\epsilon)<(1+\epsilon)
\dist{x}{\Theta(\bar p)}<r$, where $r$ is as in $(\ref{in:usregdef})$.
Thus, one obtains
$$
    \dist{\bar p}{\Theta^{-1}(x)}\le\kappa d(x,z_\epsilon)<
    \kappa (1+\epsilon)\dist{x}{\Theta(\bar p)},
$$
and hence, by arbitrariness of $\epsilon$, $(\ref{in:altusregdef})$
is satisfied with $\delta=r/2$. Conversely, since for every $z
\in\Theta(\bar p)$ it is $\dist{x}{\Theta(\bar p)}\le d(x,z)$,
then from condition $(\ref{in:altusregdef})$ one gets immediately the
validity of $(\ref{in:usregdef})$, with $r=\delta$.

Of course, whenever $\Theta$ is single-valued at $\bar p$,
uniform hemiregularity reduces to basic hemiregularity.
\end{remark}

The property of hemiregularity of $\Theta$ at $(\bar p,\bar x)$
is clearly obtained by weakening the well-known notion of
metric regularity of $\Theta$ at $(\bar p,\bar x)$, which postulates
the existence of positive reals $\kappa$ and $r$ such that
\begin{equation}     \label{in:defmr}
   \dist{p}{\Theta^{-1}(x)}\le\kappa \dist{x}{\Theta(p)},\quad\forall
   p\in\ball{\bar p}{r},\quad\forall x\in\ball{\bar x}{r}
\end{equation}
(see \cite{DonRoc14,KlaKum02,Mord06,RocWet98}).
This is readily done by fixing $p=\bar p$ in inequality $(\ref{in:defmr})$.
The following example shows that the resulting property
is actually weaker than metric regularity.

\begin{example} (A mapping which is hemiregular, whereas not metrically regular)
Let $P=\R^2$ and $X=\R$ be endowed with their usual Euclidean metric
structure. Consider the function $\Theta:\R^2\longrightarrow\R$ defined by
\begin{eqnarray*}
     \Theta(p_1,p_2)=\left\{\begin{array}{ll}
                                                 p_1+p^2_2, & \hbox{ if } p_1\ge 0, \\
                                                 p_1-p^2_2, & \hbox{ if } p_1<0,
                                               \end{array}\right.
\end{eqnarray*}
with reference point $\bar p=(0,0)$ and $\bar x=0$. $\Theta$ is not metrically
regular around $((0,0),0)$, inasmuch as, for any fixed $\kappa>0$ and $r>0$,
by taking $p=(0,\xi)$, with $0<\xi<\min\{r,\kappa^{-1}\}$, and $x=0$, the
inequality
$$
    \dist{p}{\Theta^{-1}(x)}=\xi\le\kappa\xi^2=\kappa\dist{x}{\Theta(p)}
$$
is evidently false. Nevertheless $\Theta$ turns out to be hemiregular
at the same reference pair. Indeed, for any $\kappa\ge 1$ and $r>0$,
as for every $x\in [-r,r]$ one has $(x,0)\in\Theta^{-1}(x)$, one obtains
$$
    \dist{(0,0)}{\Theta^{-1}(x)}\le |x|\le\kappa |x|=\kappa d(x,0),
$$
so that $\usreg{\Theta}{(0,0)}\le 1$.
\end{example}

Analogously, uniform hemiregularity of $\Theta$ at $(\bar p,\bar x)$
can be obtained by weakening a uniform variant of metric regularity
considered in \cite{Uder15}, which requires the existence of positive
reals $\kappa$ and $\delta$ such that
\begin{equation}     \label{in:defumr}
   \dist{p}{\Theta^{-1}(x)}\le\kappa \dist{x}{\Theta(p)},\quad\forall
   p\in\ball{\bar p}{r},\quad\forall x\in\ball{\Theta(\bar p)}{\delta}
\end{equation}
(see Definition 2.2 \cite{Uder15}). To see this, it suffices to
fix $p=\bar p$ and to notice that, if $(\bar p,\bar x)\in\grph\Theta$,
then $\ball{\bar x}{\delta}\subseteq\ball{\Theta(\bar p)}{\delta}$.
It follows that any criterion for $(\ref{in:defumr})$ to hold becomes
a sufficient condition for uniform hemiregularity. Some result of
this kind can be found in \cite{Uder15}. In particular, as a
consequence of Proposition 2.2 in \cite{Uder15}, whenever $\Theta:
P\rightrightarrows X$ is a convex process with closed graph between
Banach spaces, i.e. $\grph\Theta$ is a closed convex cone in
$P\times X$, and the following condition holds
\begin{equation}     \label{in:conproumr}
   \|\Theta^{-1}\|^-=\sup_{x\in\B} \inf_{p\in\Theta^{-1}(x)}
   \|p\|=\sup_{x\in\B}\dist{\nullv}{\Theta^{-1}(x)}<+\infty,
\end{equation}
where $\|\cdot\|$ denotes the norm on $P$, $\nullv$ stands
for the null vector of $P$ and $\B=
\ball{\nullv}{1}$, then $\Theta$ is also uniformly hemiregular
at any point $(\bar p,\bar x)\in\grph\Theta$, with
the following estimate
$$
   \usreg{\Theta}{(\bar p,\bar x)} \le\|\Theta^{-1}\|^-.
$$

\begin{remark}
Since any linear bounded operator $\Lambda:P\longrightarrow X$ between
Banach spaces, which is onto, is a convex process with closed graph satisfying
condition $(\ref{in:conproumr})$, then $\Lambda$ is also uniformly
hemiregular at each pair $(\bar p,\Lambda\bar p)$, with
$$
   \usreg{\Lambda}{(\bar p,\Lambda\bar p)} \le\|\Lambda^{-1}\|^-.
$$
As uniform hemiregularity implies hemiregularity, notice that
from the above fact it is possible to derive Proposition 5.2
in \cite{AraMor11}.
\end{remark}

Convex processes satisfying condition $(\ref{in:conproumr})$ and, as a
special case, surjective linear bounded operators, provide examples
of mappings which are uniformly hemiregular. Below, an example
is proposed of a uniformly hemiregular mapping, which fails to be
metrically regular in the sense of Definition 2.2 in \cite{Uder15}.

\begin{example} (A mapping failing to be
``uniformly metrically regular", yet uniformly hemiregular)
Let $P=\R$ and $X=\R^2$ be endowed with their usual Euclidean metric
structure. Consider the set-valued mapping $\Theta:\R\rightrightarrows\R^2$
defined by
$$
  \Theta(p)=\{x=(x_1,x_2)\in\R^2:\ x_1x_2=p\},
$$
and $\bar p=0$ and $\bar x=(0,0)$. In Example 2.2 in \cite{Uder15}
$\Theta$ has been shown to do not satisfy condition $(\ref{in:defumr})$.
Nonetheless $\Theta$ is uniformly hemiregular at $(0,(0,0))$, with
$\usreg{\Theta}{(0,(0,0))}\le 1$. Indeed, take $\delta=1$, so
that for every $x=(x_1,x_2)\in\ball{(0,0)}{1}$ one has $|x_1|\le 1$
and $|x_2|\le 1$. Since it is $\Theta^{-1}(x)=\{x_1x_2\}$, one obtains
$$
   \dist{0}{\Theta^{-1}(x)}=|x_1x_2|\le\min\{|x_1|,\, |x_2|\}=
   \dist{x}{\Theta(0)},\quad\forall x\in\ball{(0,0)}{1}.
$$
Therefore, inequality $(\ref{in:altusregdef})$, and hence Definition
\ref{def:usreg} (ii), are fulfilled with $\delta=\kappa=1$.
\end{example}

In Section \ref{Sect:1} it has been mentioned that the hemiregularity of
a set-valued mapping $\Theta:P\rightrightarrows X$ at $(\bar p,\bar x)$
can be characterized as Lipschitz lower semicontinuity property of its
inverse $\Theta^{-1}:X\rightrightarrows P$ at $(\bar x,\bar p)$
(see, for instance, \cite{AraMor11,Krug09,KruTha15}).
Recall that a set-valued mapping $\Phi:X\rightrightarrows P$ is said to be
{\it Lipschitz l.s.c.} at $(\bar x,\bar p)\in\grph\Phi$ if
there exist positive $\delta$ and $l$ such that
$$
   \Phi(x)\cap\ball{\bar p}{ld(x,\bar x)}\ne\varnothing,\quad\forall x
    \in\ball{\bar x}{\delta}.
$$

An analogous characterization can be established in the case of
uniform hemiregularity, provided that the Lipschitz lower semicontinuity
of the inverse is enhanced as follows:
a set-valued mapping $\Phi:X\rightrightarrows P$
is said to be {\it uniformly Lipschitz l.s.c.} at $(\bar x,\bar p)$
if there exist positive $\delta$ and $\l$ such that
\begin{eqnarray}     \label{def:uLiplsc}
   \Phi(x)\cap\ball{\bar p}{l\dist{x}{\Phi^{-1}(\bar p)}}\ne\varnothing,
   \quad\forall x\in\ball{\bar x}{\delta}.
\end{eqnarray}
The value
$$
    \uLlsc{\Phi}{(\bar x,\bar p)}=\inf\{l>0:\ \exists r>0 \hbox{ for which
    $(\ref{def:uLiplsc})$ holds\,} \}
$$
is called the {\it modulus of uniform Lipschitz lower semicontinuity}
of $\Phi$ at $(\bar x,\bar p)$.

\begin{proposition}      \label{pro:Liplscuhreg}
Let $\Theta:P\rightrightarrows X$ be a set-valued mapping between
metric spaces. $\Theta$ is uniformly hemiregular at $(\bar p,\bar x)
\in\grph\Theta$ iff $\Theta^{-1}$
is uniformly Lipschitz l.s.c. at $(\bar x,\bar p)$. Moreover, it
holds
$$
  \usreg{\Theta}{(\bar p,\bar x)}=
  \uLlsc{\Theta^{-1}}{(\bar x,\bar p)}.
$$
\end{proposition}

\begin{proof}
The thesis is a straightforward consequence of the above definitions
and of inequality $(\ref{in:altusregdef})$.
\end{proof}

The above characterization will be conveniently employed in the
proof of the implicit multifunction theorem presented in
Section \ref{Sect:4}.

\begin{remark}
It is useful to observe that the Lipschitz lower semicontinuity
of a mapping $\Phi:X\rightrightarrows P$, which is single-valued
in a neighbourhood of a point $\bar x\in X$, reduces to calmness
at that point, i.e. there exist positive $\delta$ and $l$ such that
$$
  \Phi(x)\in\ball{\Phi(\bar x)}{ld(x,\bar x)},\quad\forall x\in
  \ball{\bar x}{\delta}.
$$
Therefore, whenever a hemiregular set-valued mapping admits
an inverse which is locally single-valued, the latter turns out
to be calm.
\end{remark}

Metric regularity as well as many of its variants are known to admit
also characterization in terms of local surjection (openness)
properties. This is true also for uniform hemiregularity, whose surjective
behaviour is described in the next proposition.

\begin{proposition}
Let $\Theta:P\rightrightarrows X$ be a set-valued mapping between
metric spaces and let $(\bar p,\bar x)\in\grph\Theta$.

(i) If $\Theta$ is uniformly hemiregular at $(\bar p,\bar x)$ with modulus
$\usreg{\Theta}{\bar p}<+\infty$, then for any $0<a<{1\over
\usreg{\Theta}{\bar p}}$, there exists $\tilde\delta>0$ such that
\begin{eqnarray}   \label{inc:loatp}
    \Theta(\ball{\bar p}{r})\supseteq\ball{\Theta(\bar p)
    \cap\ball{\bar x}{\tilde\delta}}{ar},
    \quad\forall r\in [0,\tilde\delta).
\end{eqnarray}

(ii) If there exist positive reals $a$ and $\tilde\delta$ such that
inclusion $(\ref{inc:loatp})$ is satisfied, then $\Theta$ is
uniformly hemiregular at $(\bar p,\bar x)$ with modulus
$\usreg{\Theta}{\bar p}\le 1/a$.
\end{proposition}

\begin{proof}
(i) According to the equivalent reformulation of uniform
hemiregularity given in Remark \ref{rem:equivrefuhemic},
for any fixed $\kappa$ such that $\usreg{\Theta}{\bar p}<
\kappa<1/a$, there exists $\delta>0$
such that inequality $(\ref{in:altusregdef})$ holds. Then, set
$\tilde\delta=\delta\kappa$ and take arbitrary $r\in [0,\tilde\delta)$
and $x\in \ball{\Theta(\bar p)}{ar}\cap\ball{\bar x}{\tilde\delta}$.
Notice that, with that choice of constants, one has
$$
  \dist{x}{\Theta(\bar p)}\le {\tilde\delta\over\kappa}=\delta.
$$
Thus, inequality $(\ref{in:altusregdef})$ applies, that is
$$
  \dist{\bar p}{\Theta^{-1}(x)}\le\kappa\dist{x}{\Theta(\bar p)}
  \le\kappa ar<r.
$$
This entails that there exists $p\in P$ such that $x\in\Theta(p)$
and $p\in\ball{\bar p}{r}$, what gives that $x\in\Theta
(\ball{\bar p}{r})$. This shows the first assertion in the thesis.

(ii) Assume now inclusion $(\ref{inc:loatp})$ to hold with
positive real $a$ and $\tilde\delta$. Define $\kappa=1/a$ and
take $\delta< a\tilde\delta$. Whenever $x$ is an arbitrary element
of the set $\ball{\Theta(\bar p)}{\delta}\cap\ball{\bar x}{\delta}$,
letting $r=\dist{x}{\Theta(\bar p)}$, one has $r/a\in [0,\tilde\delta)$.
Since, according to the assumption, it is
$$
   \ball{\Theta(\bar p)\cap\ball{\bar x}{\delta}}{a{r\over a}}
   \subseteq\Theta(\ball{\bar p}{r/a}),
$$
there exists $p\in\ball{\bar p}{r/a}$ such that $x\in\Theta(p)$.
Thus, one obtains
$$
   \dist{\bar p}{\Theta^{-1}(x)}\le \dist{\bar p}{p}\le
   {r\over a}=\kappa\dist{x}{\Theta(\bar p)}.
$$
The last inequality, which is valid for every $x\in
\ball{\Theta(\bar p)}{\delta}\cap\ball{\bar x}{\delta}$, completes
the proof.
\end{proof}

\vskip1cm


\section{Uniform hemiregularity and exact penalization}
\label{Sect:3}

Let us consider a constrained optimization problem of the general
form
$$
    \min\varphi(x) \quad\hbox{ subject to }\quad
   x\in R,    \leqno ({\mathcal P})
$$
where $\varphi:X\longrightarrow\R\cup\{\pm\infty\}$ is the objective
function and $R$ denotes the feasible region, that throughout
the paper is assumed to be a nonempty closed set. The basic idea
of penalty methods consists in seeking solutions to $({\mathcal P})$
by solving unconstrained optimization problems, whose objective
function is formed by adding to $\varphi$ a term measuring the
constraint violation (see \cite{Erem67,Zang67,Zasl10}).
Since $R$ is closed, one possible representation
of the geometric constraint set $R$ is as $R=\{x\in X:\
\dist{x}{R}\le 0\}$. Consequently, one way of implementing penalty
methods is to consider the unconstrained problems
$$
  \min_{x\in\ X} \ [\varphi(x)+l\dist{x}{R}], \leqno ({\mathcal P}_l)
$$
with $l>0$. Letting $\varphi_l=\varphi+l\dist{\cdot}{R}$, function
$\varphi_l$ is said to be exact at a local solution $\bar x\in R$
to $({\mathcal P})$ provided that $\bar x$ is also a local solution
to problem $({\mathcal P}_l)$. Thus, one is interested in establishing
conditions under which $\varphi_l$ is exact, for some $l$.
It is well know that, whenever $\varphi$ is locally Lipschitz
with constant $\kappa$ at $\bar x$, then $\varphi_l$ is exact at the
same point, for every $l>\kappa$ (see, for instance \cite{Clar83}).
This fact can be taken as a starting point for developing applicable
optimality conditions for $({\mathcal P})$, especially with the aid
of nonsmooth analysis tools.
When, as it often happens in concrete applications,
$R$ is defined by specific constraints (such as inequality/equality
constraints, variational/equilibrium conditions, and so on) some further
conditions are employed to replace the geometric penalty term
$\dist{x}{R}$ by verifiable measures of the constraint violation,
called error bounds, which are expressed in terms of problem
data.

The aforementioned exactness condition comes quite expected, inasmuch as
it links the behaviour of $\varphi$ with that of function $x\mapsto
\dist{x}{R}$, which is Lipschitz continuous, indeed.
If $\varphi$ fails to be locally Lipschitz the above approach must
be modified, but its spirit can be somehow maintained by introducing
an additional assumption called problem calmness (see
\cite{Burk91,Clar76,Uder10}). This notion
requires to embed the given problem $({\mathcal P})$ in a class
of parametric optimization problems, whose feasible region comes
to depend on a parameter $p$ varying in a metric space $(P,d)$,
and then to postulate a controlled behaviour for the variations of
$\varphi$ near $\bar x$, with respect to parameter (and hence
feasible region) variations. Here, fixed a reference element
$\bar p\in P$, a set-valued mapping ${\mathcal R}:P\rightrightarrows
X$ is meant to be a parameterization of $R$ near $(\bar p,\bar x)$
provided that it fulfils the following two requirements

\begin{itemize}

\item[(i)] ${\mathcal R}(\bar p)=R$;

\item [(ii)] there exist $r>0$ and $\tau_0>0$ such that
$$
   \forall\tau\in (0,\tau_0)\ \exists p_\tau\in\ball{\bar p}{\tau}
   \backslash\{\bar p\}
   \hbox{ such that } {\mathcal R}(p_\tau)\cap\ball{\bar x}{r}
   \ne\varnothing.
$$
\end{itemize}
A given parameterization ${\mathcal R}:P\rightrightarrows
X$ of $R$ near $(\bar p,\bar x)$ enables one to define the
related family of parametric optimization problems
$$
    \min\varphi(x) \quad\hbox{ subject to }\quad
   x\in{\mathcal R}(p)    \leqno ({\mathcal P}_p)
$$
embedding $({\mathcal P})$, in the sense that for $p=\bar p$ one
obtains $({\mathcal P})$ as a special case.

\begin{definition}     \label{def:pcalm}
Given a problem $({\mathcal P})$, let ${\mathcal R}:P\rightrightarrows
X$ be a parameterization of $R={\mathcal R}(\bar p)$ near
$(\bar p,\bar x)$, where $\bar x$ is a local minimizer of $({\mathcal P})$.
Problem $({\mathcal P})$ is called {\it calm} at $\bar x$ with respect to
${\mathcal R}$ if there exist positive $r$ and $\zeta$
such that
\begin{eqnarray}   \label{in:pcalmdef}
   \varphi(x)\ge\varphi(\bar x)-\zeta d(p,\bar p),\quad\forall
   x\in\ball{\bar x}{r}\cap{\mathcal R}(p),\ \forall
   p\in\ball{\bar p}{r}.
\end{eqnarray}
The value
$$
    \pcalm{\mathcal P}{\mathcal R}{\bar x}=\inf\{\zeta>0:
   \exists r>0\hbox{ for which  $(\ref{in:pcalmdef})$ holds } \}
$$
is called {\it modulus of problem calmness} of $({\mathcal P})$
at $\bar x$, with respect to ${\mathcal R}$.
\end{definition}

Roughly speaking, the concept of problem calmness captures a
suitable interplay that intertwines the ``not optimal behaviour"
of $\varphi$ out from the feasible region of $({\mathcal P})$
and the perturbation behaviour of a parameterization of $R$
near $\bar x$, as $p$ approaches $\bar p$. This fact is illustrated
in a very simple case through the next example.

\begin{example}     \label{ex:calmparnoLip}
Consider a problem $({\mathcal P})$ defined by $X=\R$, $R=(-\infty,0]$,
and $\varphi:\R\longrightarrow\R$, given by
\begin{eqnarray*}
      \varphi(x)=\left\{\begin{array}{ll}
           \sqrt{-x}, & \hbox{ if } x\le 0, \\
           -\sqrt{x}, & \hbox{ if } x>0.
      \end{array}\right.
\end{eqnarray*}
It is evident that $\bar x=0$ is a (global) solution to  $({\mathcal P})$.
Letting $P=\R$ equipped with its usual Euclidean metric
and let $\bar p=0$, consider the parameterization ${\mathcal R}_\beta:
\R\rightrightarrows\R$ defined by
$$
    {\mathcal R}_\beta(p)=(-\infty,|p|^\beta],\quad\beta>0.
$$
Taking $x_p=|p|^\beta$, with $|p|<r$, one easily finds
$$
   {\inf_{x\in {\mathcal R}_\beta(p)}\varphi(x)
    -\varphi(\bar x)\over |p|}=
   {\varphi(x_p)-\varphi(\bar x)\over |p|}=-|p|^{{\beta\over 2}-1},
   \quad\forall p\in\R\backslash\{0\}.
$$
Therefore, according to Definition \ref{def:pcalm},  $({\mathcal P})$ turns
out to be calm at $\bar x$ with respect to ${\mathcal R}_\beta$
iff $\beta\ge 2$. Notice that $\varphi$ is not locally Lipschitz
at $0$.
\end{example}

Once a parameterization of $R$ has been defined, the related
notion of problem calmness allows one to establish
an exact penalization result by introducing the following
penalty functions
$$
    \varphi_l(p,x)=\varphi(x)+l\dist{x}{{\mathcal R}(p)}.
$$
In this concern, the property of uniform hemiregularity of
${\mathcal R}$ at $\bar p$ plays an essential role, as it
appears from the below result.

\begin{theorem}     \label{thm:exctpen}
Let $\bar x\in R$ be a local solution to  $({\mathcal P})$ and let
${\mathcal R}:P\rightrightarrows X$ be a parameterization of $R$
at $(\bar p,\bar x)$, with $\bar p\in P$ being a reference value. If

\begin{enumerate}

\item[(i)] ${\mathcal R}:P\rightrightarrows X$ is uniformly hemiregular
at $\bar p$;

\item[(ii)] $({\mathcal P})$ is calm at $\bar x$ with respect to ${\mathcal R}$;

\end{enumerate}
then, function $\varphi_l(\bar p,\cdot)$ is exact at $\bar x$ for every
$l>\usreg{\mathcal R}{\bar p}\cdot\pcalm{\mathcal P}{\mathcal R}{\bar x}$.
\end{theorem}

\begin{proof}
Fix an abitrary $l$, with $l>\usreg{\mathcal R}{\bar p}\cdot\pcalm{\mathcal P}
{\mathcal R}{\bar x}$. Then, according to Definition \ref{def:usreg}(ii)
and Definition \ref{def:pcalm}, it is possible to pick $\kappa>
\usreg{\mathcal R}{\bar p}$, $\zeta>\pcalm{\mathcal P}{\mathcal R}
{\bar x}$ and $\epsilon>0$ such that:
\begin{itemize}

\item for some $r_1>0$ it holds
\begin{eqnarray}       \label{in:usreghypi}
   \dist{\bar p}{{\mathcal R}^{-1}(x)}\le\kappa
   \dist{x}{{\mathcal R}(\bar p)},\quad\forall
   x\in\ball{{\mathcal R}(\bar p)}{r_1};
\end{eqnarray}

\item for some $r_2>0$ it holds
\begin{eqnarray}     \label{in:pcalmhypii}
 \varphi(x)\ge\varphi(\bar x)-\zeta d(p,\bar p),\quad\forall
    x\in\ball{\bar x}{r_2}\cap{\mathcal R}(p),\ \forall
   p\in\ball{\bar p}{r_2};
\end{eqnarray}

\item it is
\begin{eqnarray}     \label{in:kappazetaeps}
    l>\kappa(1+\epsilon)\zeta.
\end{eqnarray}
\end{itemize}
Ab absurdo, let us suppose that $\varphi_l(\bar p,\cdot)$ fails to be
exact. This means that for every $n\in\N$ there exists $x_n\in\ball{\bar x}
{1/n}$ such that
\begin{eqnarray} \label{in:absnegexact}
     \varphi(x_n)+l\dist{x_n}{{\mathcal R}(\bar p)}<\varphi(\bar x).
\end{eqnarray}
Since $\bar x$ is a local solution to $({\mathcal P})$ and $x_n\to
\bar x$ as $n\to\infty$, there must exist $\bar n\in\N$ such that
$x_n\not\in{\mathcal R}(\bar p)=R$ for every $n\in\N$, with $n\ge
\bar n$. Consequently, as ${\mathcal R}(\bar p)$
is a closed set, one has
$$
    \dist{x_n}{{\mathcal R}(\bar p)}>0,\quad\forall n\in\N,\ n\ge\bar n.
$$
On the other hand, as $\bar x\in{\mathcal R}(\bar p)$, one has
$$
    \dist{x_n}{{\mathcal R}(\bar p)}\le d(x_n,\bar x)\le{1\over n},
$$
whence
\begin{eqnarray*}
    \lim_{n\to\infty}\dist{x_n}{{\mathcal R}(\bar p)}=0.
\end{eqnarray*}
Thus, by increasing the value of $\bar n$ if needed, one obtains
$x_n\in\ball{{\mathcal R}(\bar p)}{r_1}\backslash{\mathcal R}(\bar p)$
and hence, according to $(\ref{in:usreghypi})$, it must be
$$
    \dist{\bar p}{{\mathcal R}^{-1}(x_n)}\le\kappa
    \dist{x_n}{{\mathcal R}(\bar p)},\quad\forall
    n\in\N,\ n\ge\bar n.
$$
This means that for every $n\in\N$, with $n\ge\bar n$, there exists
$p_n\in {\mathcal R}^{-1}(x_n)$ such that
\begin{eqnarray}      \label{in:distpnbarp}
  d(p_n,\bar p)<\kappa(1+\epsilon)\dist{x_n}{{\mathcal R}(\bar p)},
\end{eqnarray}
where $\epsilon$ is as in inequality $(\ref{in:kappazetaeps})$.
Notice that, as $x_n\in{\mathcal R}(p_n)$ and $x_n\not\in
{\mathcal R}(\bar p)$, it has to be $p_n\ne\bar p$.
From inequalities $(\ref{in:absnegexact})$ and $(\ref{in:distpnbarp})$,
it follows
$$
   {\kappa(1+\epsilon)\over d(p_n,\bar p)}[\varphi(x_n)-\varphi(\bar x)]<
   {\varphi(x_n)-\varphi(\bar x)\over\dist{x_n}{{\mathcal R}
   (\bar p)}}<-l,
$$
whence, on account of inequality $(\ref{in:kappazetaeps})$,
one obtains
$$
   \varphi(x_n)<\varphi(\bar x)-{l\over\kappa(1+\epsilon)}
  d(p_n,\bar p)<-\zeta d(p_n,\bar p).
$$
Since $p_n\to\bar p$ as $n\to\infty$ because of $(\ref{in:distpnbarp})$,
by increasing further the value of $\bar n\in\N$, if needed, one finds that
$x_n\in\ball{\bar x}{r_2}\cap{\mathcal R}(p_n)$ and $p_n\in
\ball{\bar p}{r_2}$. Therefore, the last inequality contradicts
inequality $(\ref{in:pcalmhypii})$. This completes the proof.
\end{proof}

\begin{remark}
It is to be noted that the above theorem can be also derived as a
special case from a more general theorem, which was recently established
within a unifying approach to the theory of exactness in
penalization methods (see \cite[Theorem 2.12]{Dolg16}). Nevertheless,
in formulating that theorem, the uniform hemicontinuity is
not mentioned and its role remains hidden, because the mere
topological space setting, where optimization problems are
considered, does not allow to do so. Moreover, some extra
assumptions enter the statement of that result. Theorem
\ref{thm:exctpen} is therefore a refinement of a special case
of Theorem 2.12, whose self-contained proof here proposed
emphasizes the role of the main property under study.
\end{remark}

It is worth mentioning that in the original definition of problem
calmness the parameter $p$ was supposed to perturb linearly the
constraining mappings (see \cite{Burk91,Clar76}). In that special
case, it was possible to fully characterize the exactness of penalty
functions by means of the resulting notion of problem calmness, what
does not remain true for perturbations of more general type
(see \cite{Uder10}). Thus
the above result is complemented here with a result providing a sufficient
condition, upon which problem $({\mathcal P}_{\bar p})$ turns out
to be calm with respect to a given parameterization.

\begin{proposition}
With reference to a problem parameterization $({\mathcal P}_p)$,
let $\bar x\in{\mathcal R}(\bar p)$ be a local minimizer of
$({\mathcal P}_{\bar p})$, with $\bar p\in P$.
Suppose that

\begin{enumerate}

\item[(i)] ${\mathcal R}$ is calm at $(\bar p,\bar x)$, i.e.
there exist positive reals $\zeta$ and $r$ such that
\begin{equation}     \label{in:calmsetvmap}
  {\mathcal R}(p)\cap\ball{\bar x}{r}\subseteq
  \ball{{\mathcal R}(\bar p)}{\zeta\dist{p}{\bar p}},
  \quad\forall p\in \ball{\bar p}{r}.
\end{equation}

\item[(ii)] there exists $l>0$ such that $\varphi_l(\bar p,\cdot)$
is exact at $\bar x$.
\end{enumerate}
Then, problem $({\mathcal P}_{\bar p})$ is calm at $\bar p$
with respect to ${\mathcal R}$.
\end{proposition}

\begin{proof}
Assume, ab absurdo, that for every $n\in\N$ there exist
$p_n\in\ball{\bar p}{1/n}\backslash\{\bar p\}$ and $x_n
\in {\mathcal R}(p_n)\cap\ball{\bar x}{1/n}$ such that
\begin{eqnarray}     \label{in:prcalmneg}
  \varphi(x_n)<\varphi(\bar x)-n\dist{p_n}{\bar p}.
\end{eqnarray}
Since ${\mathcal R}$ is supposed to be calm at $(\bar p,\bar x)$,
there exist positive reals $\zeta$ and $r$ such that
inclusion $(\ref{in:calmsetvmap})$ holds true.
By virtue of this inclusion,
the fact that $p_n$ converges to $\bar p$ and $x_n$ converges
to $\bar x$ as $n\to +\infty$ implies that $x_n\in
\ball{{\mathcal R}(\bar p)}{\zeta\dist{p_n}{\bar p}}$,
so that one obtains
$$
  \dist{x_n}{{\mathcal R}(\bar p)}\le\zeta\dist{p_n}{\bar p}.
$$
Consequently, from inequality $(\ref{in:prcalmneg})$ it follows
$$
   \varphi(x_n)<\varphi(\bar x)-{n\over\zeta}
   \dist{x_n}{{\mathcal R}(\bar p)},
$$
which evidently contradicts hypothesis (ii).
\end{proof}


\section{An implicit function theorem for uniform hemiregularity}
\label{Sect:4}

In the main result of the previous section, the exact penalization
of a constrained optimization problem is obtained upon a uniform
hemiregularity assumption on a parameterization of its feasible
region. In order to make viable such an approach, conditions are
needed, which can guarantee a given parameterization to be uniformly
hemiregular. This issue is considered in the present section in
the case of feasible regions defined by an abstract equilibrium constraint,
namely by constraints of the form
$$
  \omega\in\Phi(x),  \leqno  ({\mathcal E})
$$
where $\Phi:X\rightrightarrows Y$ is a given set-valued mapping
between metric spaces and $\omega$ is a given element of $Y$.
The format of problem $({\mathcal E})$ is general enough to cover the constraint
systems mostly occurring in the mainly investigated optimization
problems, such as equality/inequality systems, cone constraints,
equilibrium conditions, generalized equations, lower level optimality
in hierarchic optimization problems, and so on.
In order to define a parameterization of the solution set of
problem $({\mathcal E})$, one may consider the following problem
perturbation, which is defined via any set-valued
mapping $F:P\times X\rightrightarrows Y$, such that $F(\bar p,x)=
\Phi(x)$ for every $x\in X$:
$$
    \omega\in F(p,x).    \leqno  ({\mathcal E}_p)
$$
The solution mapping associated with $({\mathcal E}_p)$ is
therefore given by
$$
  {\mathcal R}(p)=F^{-1}(p,\cdot)(\omega)=
  \{x\in X:\ \omega\in F(p,x)\}.
$$
It is clear that an analytical expression of the (generally) set-valued
mapping ${\mathcal R}$ can be hardly derived from $({\mathcal E}_p)$
by direct computations, because of the severe difficulties in solving
explicitly each problem $({\mathcal E}_p)$. Therefore, it is convenient
to investigate the hemiregularity property of ${\mathcal R}$ via
an implicit multifunction theorem. Such a task is carried out in what
follows by a variational technique. To this aim, let us denote
by $\dispo{F}:P\times X\longrightarrow [0,+\infty]$ the following
functional quantifying the the displacement of $F$ from $\omega$:
$$
    \disp{F}{p}{x}=\dist{\omega}{F(p,x)}.
$$
In view of a subsequent employment,
a first semicontinuity property of $\dispo{F}(\cdot,x)$ is stated
in the next technical lemma, that can be easily obtained as a
special case of \cite[Lemma 3.2]{Uder14}.

\begin{lemma}     \label{lem:lscdisp}
Let $F:P\rightrightarrows Y$ be a set-valued mapping between
metric spaces and let $\omega\in Y$.
If $F$ is Hausdorff u.s.c. at $\bar p\in\dom F$, i.e. for every
$\epsilon>0$ there exists $\delta_\epsilon>0$ such that
$$
   F(p)\subseteq\ball{F(\bar p)}{\epsilon},\quad\forall p\in
   \ball{\bar p}{\delta_\epsilon},
$$
then the function $p\mapsto\dist{\omega}{F(p)}$ is l.s.c. at $\bar p$.
\end{lemma}

\begin{remark}      \label{rem:uscHusc}
In the sequel, the fact will be exploited that the thesis of
Lemma \ref{lem:lscdisp} is true a fortiori if is $F$ is u.s.c..
Indeed, the (merely topological) notion of upper semicontinuity
at a point implies Hausdorff upper semicontinuity at the same point.
\end{remark}

For the purposes of the present analysis, the continuity
properties of the function $\dispo{F}$ are not enough.
Derivative-like tools, that enable one to formulate
conditions generalizing the nonsingularity requirement
in the classical implicit function theorem, are actually
needed. In a purely metric space setting, such tools are
mainly based on the notion of strong slope (see \cite{DeMaTo80}).
More precisely, a more robust variant of it, called strict
outer slope, will be employed here in connection with the
displacement function, which is defined as follows:

\begin{eqnarray*}
   \sostslp{\dispo{F}}(\bar p)=\lim_{\epsilon\to 0^+}\inf\
   \{\stslp{\dispo{F}}(p,x):\ p\in\ball{\bar p}{\epsilon},\
   x\in \ball{\bar x}{\epsilon}, \\
      \disp{F}{\bar p}{\bar x}<\disp{F}{p}{ x}<
      \disp{F}{\bar p}{\bar x}+\epsilon\},
\end{eqnarray*}
where
$$
    \stslp{\dispo{F}}(p,x)=\left\{ \begin{array}{ll}
                       0, & \text{if $p$ is a local minimizer} \\
                           & \text{to $\dispo{F}(\cdot,x)$},  \\
                       \displaystyle\limsup_{q\to p}
                       \frac{\dispo{F}(p,x)-\dispo{F}(q,x)}{d(q,p)}, &
                       \text{otherwise,}
                 \end{array}
   \right.
$$
is the partial strong slope of function $\dispo{F}$ with respect to
the variable $p$, calculated at $(p,x)\in P\times X$.
For more details on this slope as well as on other variations
on this theme, the reader is refereed, for instance, to
\cite{FaHeKrOu10}.

Now, all the needed elements having been introduced, the main result
of the paper can be formulated.

\begin{theorem}    \label{thm:imthreg}
Let $F:P\times X\rightrightarrows Y$ be a set-valued mapping
defining a problem perturbation $({\mathcal E}_p)$, with solution
mapping ${\mathcal R}:P\rightrightarrows X$. Given $\bar p\in P$,
let $\bar x\in {\mathcal R}(\bar p)$. Suppose that:

\begin{enumerate}

\item[(i)] $(P,d)$ is metrically complete;

\item[(ii)] there exists $\delta_0>0$ such that $F(\cdot,x):P
\rightrightarrows Y$ is Hausdorff u.s.c. on $\ball{\bar p}{\delta_0}$,
for every $x\in\ball{\bar x}{\delta_0}$;

\item[(iii)] $F(\bar p,\cdot):X\rightrightarrows Y$ is uniformly
Lipschitz l.s.c. at $(\bar x,\omega)$;

\item[(iv)] $\sostslp{\dispo{F}}(\bar p)>0$.

\end{enumerate}
Then, ${\mathcal R}$ is uniformly hemiregular at $(\bar p,\bar x)$ and
the following estimate holds
\begin{equation}    \label{in:uhregmodest}
    \usreg{\mathcal R}{(\bar p,\bar x)}\le {\uLlsc{F}{(\bar x,\omega)}
    \over \sostslp{\dispo{F}}(\bar p)}.
\end{equation}
\end{theorem}

\begin{proof}
According to hypothesis (iv), it is possible to pick a constant
$\alpha$ such that
\begin{eqnarray}     \label{in:sostslposuse}
    0<\alpha<\sostslp{\dispo{F}}(\bar p).
\end{eqnarray}
As established in Proposition \ref{pro:Liplscuhreg}, hypothesis
(iii) is equivalent to suppose the mapping $F^{-1}(\bar p,\cdot)
:Y\rightrightarrows X$ to be uniformly hemiregular at $(\omega,
\bar x)$. Since the related moduli coincide,
this means that, corresponding to any $\kappa>\uLlsc{F}
{(\bar x,\omega)}$, there exists $r_\kappa>0$ such that
\begin{eqnarray}    \label{in:usregF}
  \qquad\dist{\omega}{F(\bar p,x)}\le\kappa\dist{x}{F^{-1}(\bar p,\cdot)
  (\omega)},\ \forall x\in\ball{\bar x}{r_\kappa},
\end{eqnarray}
where it is to be recalled that $F^{-1}(\bar p,\cdot)(\omega)={\mathcal R}(\bar p)$.
Define $\tilde\delta=\min\{\delta_0,r_\kappa\}$. Observe that inequality
$(\ref{in:sostslposuse})$ means that, corresponding to $\alpha$,
it is possible to find $\delta_*\in(0,\tilde\delta)$ such that
\begin{eqnarray}    \label{in:stsldispgta}
    \stslp{\dispo{F}}(p,x)>\alpha,
\end{eqnarray}
$$
   \forall p\in\ball{\bar p}{\delta_*},\
   \forall x\in\ball{\bar x}{\delta_*},\
   \hbox{ with } 0=\disp{F}{\bar p}{\bar x}<\disp{F}{p}{x}<\delta_*.
$$
In turn, the inequality $(\ref{in:stsldispgta})$ implies that, whenever
$(p,x)\in\ball{\bar p}{\delta_*}\times\ball{\bar x}{\delta_*}$,
with $0<\disp{F}{p}{x}<\delta_*$, then for every $\eta>0$ there exists
$p_\eta\in\ball{p}{\eta}$ such that
\begin{eqnarray}      \label{in:etasstsl}
    \disp{F}{p}{x}>\disp{F}{p_\eta}{x}+\alpha d(p_\eta,p).
\end{eqnarray}
Now, choose a positive real $r_*$ satisfying the following
condition
$$
    r_*<\min\left\{{\delta_*\over 2},{\delta_*\over \kappa}\right\},
$$
and fix an arbitrary $x\in\ball{\bar x}{r}\backslash
{\mathcal R}(\bar p)$, with
\begin{eqnarray}      \label{in:rusregR}
    0<r<\min\left\{r_*,{\alpha r_*\over 3\kappa}
    \right\}.
\end{eqnarray}
Let us consider the function $\disp{F}{\cdot}{x}:
\ball{\bar p}{r_*}\longrightarrow [0,+\infty]$. It is obviously
bounded from below and, since $r<r_*<\delta_*<\tilde\delta\le\delta_0$,
then, by virtue of hypothesis (ii) and Lemma \ref{lem:lscdisp},
function $\disp{F}{\cdot}{x}$ is l.s.c. on $\ball{\bar p}{r_*}$.
Owing to hypothesis (i), $\ball{\bar p}{r_*}$ turns out to be a
complete metric space. Furthermore, notice that, since
$r<r_*<\delta_*<\tilde\delta\le r_\kappa$ and hence $x\in\ball{\bar x}
{r_\kappa}$, then according to $(\ref{in:usregF})$ it holds
\begin{eqnarray*}
   \disp{F}{\bar p}{x} \le \kappa\dist{x}{{\mathcal R}(\bar p)}\le
   \inf_{p\in\ball{\bar p}{r_*}}\disp{F}{p}{x}
  +\kappa\dist{x}{{\mathcal R}(\bar p)}.
\end{eqnarray*}
By applying the Ekeland's variational principle, one obtains
the existence of an element $p_0\in \ball{\bar p}{r_*}$ such that
\begin{eqnarray}
   \disp{F}{p_0}{x}\le \disp{F}{\bar p}{x};
\end{eqnarray}
\begin{eqnarray}     \label{in:EVP2}
   d(p_0,\bar p)\le {\kappa\dist{x}{{\mathcal R}(\bar p)}
    \over\alpha};
\end{eqnarray}
\begin{eqnarray}     \label{in:EVP3}
   \disp{F}{p_0}{x}< \disp{F}{p}{x}+\alpha d(p,p_0),
   \quad\forall p\in\ball{\bar p}{r_*}\backslash\{p_0\}.
\end{eqnarray}
Let us show that the last inequalities entail that
$$
    \disp{F}{p_0}{x}=0,
$$
so that, as $F$ takes closed values, $x\in{\mathcal R}(p_0)$.
Assume, ab absurdo, that $\disp{F}{p_0}{x}>0$. Since it is
$x\in\ball{\bar x}{r_*}$ and
$$
   \disp{F}{p_0}{x}\le\disp{F}{\bar p}{x}\le\kappa
   \dist{x}{{\mathcal R}(\bar p)}\le\kappa r_*<\delta_*,
$$
one has
$$
   p_0\in \ball{\bar p}{\delta_*} \quad\hbox{ and }\quad
   x\in\ball{\bar x}{\delta_*},\ \hbox{ with }
   0<\disp{F}{p_0}{x}<\delta_*.
$$
Thus, if taking $\eta=r_*/2$, according to inequality $(\ref{in:etasstsl})$,
an element $p_\eta$ must exist in $\ball{p_0}{r_*/2}$, with $p_\eta
\ne p_0$, such that
\begin{eqnarray}    \label{in:contrEVP3}
     \disp{F}{p_0}{x}>\disp{F}{p_\eta}{x}+\alpha d(p_\eta,p_0).
\end{eqnarray}
Observe that, by virtue of inequalities  $(\ref{in:EVP2})$ and
$(\ref{in:rusregR})$, it results in
$$
    d(p_0,\bar p)\le{\kappa\over\alpha}\cdot
    {\alpha r_*\over 3\kappa}<{r_*\over 2}.
$$
As a consequence, $p_\eta$ must belong to $\ball{\bar p}{r_*}\backslash
\{p_0\}$, because it holds
$$
   d(p_\eta,\bar p)\le d(p_\eta,p_0)+d(p_0,\bar p)<{r_*\over 2}+
    {r_*\over 2}.
$$
Therefore, inequality $(\ref{in:EVP3})$ is found to be evidently
contradicted by inequality $(\ref{in:contrEVP3})$.
From the fact that  $x\in{\mathcal R}(p_0)$, by recalling once again
inequality $(\ref{in:EVP2})$, one obtains that
$$
    \dist{\bar p}{{\mathcal R}^{-1}(x)}\le d(\bar p,p_0)\le
    {\kappa\over\alpha}\dist{x}{{\mathcal R}(\bar p)}.
$$
Since by arbitrariness of $x$ the last inequality remains true all over
$\ball{\bar x}{r}$, the set-valued mapping ${\mathcal R}$
is shown to be uniformly hemiregular at $\bar p$, with $\usreg{{\mathcal R}}
{\bar p}\le\kappa/\alpha$. Since $\alpha$ and $\kappa$ can be taken
arbitrarily closed to the value of $\sostslp{\dispo{F}}(\bar p)$ and
$\usreg{F^{-1}(\bar p,\cdot)}{\omega}$, respectively, then from the
last inequality it is possible to derive the estimate appearing in
the thesis. This completes the proof.
\end{proof}

As a comment to Theorem \ref{thm:imthreg}, it is to be noted that
its thesis combines solvability and sensitivity information,
according to the spirit of implicit function theorems. Indeed,
problems $({\mathcal E}_p)$ turn out to be solvable for every
$p$ in a neighbourhood of $\bar p$, as a direct consequence
of the hemiregularity of ${\mathcal R}$ at $(\bar p,\bar x)$.
The sensitivity part comes from the estimation of
$\usreg{\mathcal R}{(\bar p,\bar x)}$, which is fully expressed
in terms of problem data.

To the best of the author's knowledge, the only existing implicit
multifunction theorem involving hemiregularity is
\cite[Theorem 5.4]{AraMor11}.
A direct comparison of Theorem \ref{thm:imthreg} with this
result can not be accomplished for several reasons.
First, even if restated in a common setting
(\cite[Theorem 5.4]{AraMor11} is valid in Banach spaces),
they consider solution mappings associated
with different problems (an inclusion problem involving a
set-valued mapping versus an equation with a perturbing
term). Besides, assuming to consider a single-valued mapping $F$
in Theorem \ref{thm:imthreg} and a null perturbation term
$g\equiv\nullv$ in Theorem 5.4, the former considers uniform
hemiregularity, whereas the latter deals with a mere
hemiregularity with respect to one variable, which is
uniform with respect to the other variable.
Nevertheless, with all that, a common
pattern can be traced: Theorem \ref{thm:imthreg} assumes
the uniform Lipschitz lower semicontinuity with respect
to $x$ of the problem data to gain the uniform hemiregularity
of the solution mapping, while Theorem 5.4 assumes the hemiregularity
with respect to $x$ of the problem data to achieve the Lipschitz
lower semicontinuity \footnote{Actually, in the statement
of Theorem 5.4 this property is not mentioned, but is expressed
as hemiregularity of the inverse multifunction.}
of the solution mapping. The condition enabling this
phenomenon is the nondegeneracy of the strict outer slope
with respect to $p$ of the displacement functional in the
first case, which is replaced by a calmness condition
with respect to $p$ in Theorem 5.4.

To assess the impact of the above result on constrained
optimization, let us consider problems of the form
$$
    \min\varphi(x) \quad\hbox{ subject to }\quad
   x\in R=\Phi^{-1}(\omega),   \leqno ({\mathcal P}_{{\mathcal E}})
$$
that is with constraints in the abstract form $({\mathcal E})$.
The reader should notice that, even though inequality $(\ref{in:altusregdef})$
involves the set ${\mathcal R}(\bar p)$, which seems to require
the knowledge of the feasible region of $({\mathcal P}_{{\mathcal E}})$,
nonetheless Theorem \ref{thm:imthreg} can be effectively exploited for
achieving the exactness of penalty functions, if combined with
problem calmness, as stated next.
Below, by penalty function $\varphi_l:P\times X\longrightarrow
\R\cup\{\pm\infty\}$ associated with problem $({\mathcal P}_{{\mathcal E}})$,
the following functional is meant:
$$
  \varphi_l(p,x)=\varphi(x)+l\dispo{F(p,x)}.
$$
Observe that in order to evaluate $\varphi_l$ one needs only
the problem data.

\begin{corollary}    \label{cor:expenimth}
Let $F:P\times X\rightrightarrows Y$ be a perturbation of
$\Phi$ defining a parameterization of the feasible reagion
of problem $({\mathcal P}_{{\mathcal E}})$, let $\bar p\in
P$ such that $F(\bar p,\cdot)=\Phi$, and let $\bar x$ be a
local solution of $({\mathcal P}_{{\mathcal E}})$. Suppose
that

\begin{enumerate}

\item[(i)] $(P,d)$ is metrically complete;

\item[(ii)] there exists $\delta_0>0$ such that $F(\cdot,x):P
\rightrightarrows Y$ is Hausdorff u.s.c. on $\ball{\bar p}{\delta_0}$,
for every $x\in\ball{\bar x}{\delta_0}$;

\item[(iii)] $F(\bar p,\cdot):X\rightrightarrows Y$ is uniformly
Lipschitz l.s.c. at $(\bar x,\omega)$;

\item[(iv)] $\sostslp{\dispo{F}}(\bar p)>0$;

\item[(v)] $({\mathcal P}_{{\mathcal E}})$ is calm at $\bar x$
with respect to the parameterization ${\mathcal R}$ defined by $F$.

\end{enumerate}
Then, for every
\begin{equation}   \label{in:expenimtcost}
  l>{\uLlsc{F}{(\bar x,\omega)}\cdot
   \pcalm{{\mathcal P}_{{\mathcal E}}}{\mathcal R}{\bar x}
    \over \sostslp{\dispo{F}}(\bar p)}
\end{equation}
function $\varphi_l(\bar p,\cdot)$ is exact at $\bar x$.
\end{corollary}

\begin{proof}
It suffices to apply Theorem \ref{thm:imthreg} and Theorem
\ref{thm:exctpen}. The estimate $(\ref{in:expenimtcost})$
can be immediately obtained by inequality $(\ref{in:uhregmodest})$
and the condition on the penalty term appearing in the thesis of
Theorem \ref{thm:exctpen}.
\end{proof}

To guide a comparison of Theorem \ref{thm:imthreg} with other
similar implicit multifunction theorems of new generation, it
must be pointed out that, often, along with the local solvability
of the parameterized system ${(\mathcal E}_p)$, a local error bound
of the form
$$
  \dist{x}{{\mathcal R}(p)}\le\kappa\dist{\omega}{F(p,x)},
$$
is also established, with $\kappa>0$ and with $p$ varying around
$\bar p$, or $p=\bar p$ (let us mention here
\cite[Theorem 5.5.5]{BorZhu05},
which served as a paradigm for many epigones in the subsequent
literature). Such distance estimates, stemming from
the Lyusternik's theorem, are useful for deriving optimality
conditions for problems with Lipschitz objective functions.
Of course, they can be generalized obtaining H\"older type
estimates in order to treat problems with corresponding  H\"older
objective functions. In contrast to this, in Corollary \ref{cor:expenimth}
no assumption is made on the objective function of problem
$({\mathcal P}_{{\mathcal E}})$, apart problem calmness (hypothesis (v)),
which relates to both $\varphi$ and ${\mathcal R}$. Thus, the
present approach to implicit multifunction theorem reveals that
Lipschitz/H\"older assumptions on the objective function can be
dropped out at the price of introducing a suitable interplay between
the parameterization of the feasible region and the objective
function.

The rest of the current section is devoted to establish a version
of Theorem \ref{thm:imthreg} working in Banach spaces.
Such a setting, which is more structured than purely metric spaces,
enables one to reformulate the condition on the strict outer
slope of the displacement functional in terms of derivative-like
objects. Since the displacement functional is rarely expected to be
differentiable, this will be done by employing tools of nonsmooth analysis.
More precisely, the partial Fr\'echet coderivative of the set-valued
mapping $F$ will be used. In order to recall this generalized derivative
construction, some basic elements of the Fr\'echet subdifferential calculus and
the related geometry are needed. In what follows, whenever $(X,
\|\cdot\|)$ denotes a Banach space, its continuous dual and
the related unit ball are indicated by $X^*$ and $\B^*$,
respectively.
Given a function $\varphi:X\longrightarrow\R\cup\{\pm\infty\}$
defined on a Banach space and $\bar x\in\dom\varphi=\{x\in\X:\
|\varphi(x)|<+\infty\}$, the Fr\'echet (alias, regular) subdifferential
of $\varphi$ at $\bar x$ is defined by
$$
   \fsubdif\varphi(\bar x)=\left\{ x^*\in X^*:\ \liminf_{x\to\bar x}
   \frac{\varphi(x)-\varphi(\bar x)-\langle x^*,x-\bar x\rangle}
   {\|x-\bar x\|}\ge 0\right\}.
$$
Given a subset $S\subseteq X$ and $\bar x\in S$, the Fr\'echet (alias,
regular) normal cone of $S$ at $\bar x$ is defined by
$$
   \Normal(\bar x,S)=\left\{x^*\in X^*:\ \limsup_{S\atop
    \displaystyle x\to\bar x}\frac{\langle x^*,x-\bar x
    \rangle}{\|x-\bar x\|}\le 0\right\}.
$$
Notice that the two aforementioned notions are linked through the set indicator
function $\iota_S:X\longrightarrow\{0,+\infty\}$, in the sense that
$$
   \Normal(\bar x,S)=\fsubdif\iota_S(\bar x),
    \quad \bar x\in X.
$$
Given a set-valued mapping $\Phi:X\rightrightarrows Y$ between
Banach spaces and $(\bar x,\bar y)\in\grph\Phi$, the Fr\'echet
coderivative of $\Phi$ at $(\bar x,\bar y)$ is the set-valued
mapping $\Coder\Phi(\bar x,\bar y):Y^*\rightrightarrows X^*$
defined through the Fr\'echet normal cone to its
graph as follows
$$
  \Coder\Phi(\bar x,\bar y)(y^*)=\{x^*\in X^*:\
  (x^*,-y^*)\in\Normal((\bar x,\bar y),\grph\Phi)\},
  \quad y^*\in Y^*.
$$
The Fr\'echet subdifferential, the Fr\'echet normal cone
and the Fr\'echet coderivative are the basic pillars of the
nonsmooth calculus here employed.
It is well known that the natural environment where to
handle the aforementioned Fr\'echet constructions are Asplund
spaces. Recall that a Banach space $(X,\|\cdot\|)$ is said
to be Asplund if every continuous convex function defined on
a nonempty open convex subset $C$ of $X$ is Fr\'echet differentiable
on a dense $G_\delta$ subset of $C$. It has been proved that the
Asplund property for a Banach space can be characterized by the
fact that each of its separable subspaces admits a separable dual
(see \cite{BorZhu05,Mord06}). The class of Asplund spaces,
including all weakly compactly generated spaces and, hence, all
reflexive Banach spaces, is wide enough for many applications.
Moreover, every Banach space having Fr\'echet smooth bump
functions (in particular, every space admitting a Fr\'echet
smooth renorm) is Asplund.
One of reasons why Asplund spaces are the natural environment
for the Fr\'echet nonsmooth calculus deals with the fact
that the Asplund property can be also characterized in terms
of the validity of the following Lipschitz local approximate
Fr\'echet subdifferential (for short, fuzzy) sum rule.
Such a rule plays a key role in many circumstances arising
in optimization and variational analysis.

\begin{definition}      \label{def:fuzsumrul}
Let $(X,\|\cdot\|)$ be a Banach space. $X$ is said to satisfy
the {\it Fr\'echet fuzzy sum rule} if for any l.s.c. function
$\varphi_1:X\longrightarrow\R\cup\{+\infty\}$, any Lipschitz
function $\varphi_2:X\longrightarrow\R$, and any $\epsilon>0$,
whenever $\bar x\in X$ is a local minimizer of $\varphi_1+
\varphi_2$, there exist $x_i\in X$ and $x^*_i\in\fsubdif\varphi_i
(x_i)$, $i=1,\, 2$, such that
$$
 (x_i,\varphi_i(x_i))\in\ball{(\bar x,\varphi_i(\bar x))}{\epsilon},
 \quad i=1,\, 2,
$$
and
$$
   \|x^*_1+x^*_2\|<\epsilon.
$$
\end{definition}
Among the notable achievements of nonlinear functional analysis,
there is the understanding that a Banach space is Asplund iff
it satisfies the Fr\'echet fuzzy sum rule (see \cite{BorZhu05,Mord06}).
In other words, any Asplund space is $\fsubdif$-trustworthy in
the sense of \cite{Ioff00}.
The next lemma adapts \cite[Proposition 1]{Ioff00} to the
specific need of the present analysis.

\begin{lemma}       \label{lem:stslpfsubdif}
Let $F:W\rightrightarrows Y$ be a set-valued mapping between
Banach spaces and let $W\subseteq P\times X$ be an open set.
Suppose that:
\begin{enumerate}

\item[(i)] $(P,\|\cdot\|)$ is Asplund;

\item[(ii)] the set-valued mapping $F(\cdot,x)$ is
Hausdorff u.s.c. on $\Pi_P(W)=\{p\in P:\
\exists x\in X:\ (p,x)\in W\}$, for each $x\in\Pi_X(W)=\{x\in X:\
\exists p\in P:\ (p,x)\in W\}$.

\end{enumerate}
Then, it holds
$$
  \inf_{(p,x)\in W}\stslp{\dispo{F}}(p,x)\ge\inf\{
  \|p^*\|:\ p^*\in\fsubdif_p{\dispo{F}}(p,x),\
  (p,x)\in W\}.
$$
\end{lemma}

\begin{proof}
Observe that, by virtue of Lemma \ref{lem:lscdisp}, each function
$p\mapsto{\dispo{F}}(p,x)$, with $x\in\Pi_X(W)$, is l.s.c.on on $\Pi_P(W)$.
Set $\mu=\inf_{(p,x)\in W}\stslp{\dispo{F}}(p,x)$ and take an arbitrary
$\epsilon>0$. Corresponding to $\epsilon/2$, there exists
$(p_\epsilon,x_\epsilon)\in W$ such that $\stslp{\dispo{F}}
(p_\epsilon,x_\epsilon)<\mu+{\epsilon\over 2}$. According to the
definition of partial strong slope, the last inequality implies
the existence of $\delta>0$ such that
$$
  {\dispo{F}}(p,x_\epsilon)+\left(\mu+{\epsilon\over 2}\right)
  \|p-p_\epsilon\|\ge{\dispo{F}}(p_\epsilon,x_\epsilon),
  \quad\forall p\in\ball{p_\epsilon}{\delta}.
$$
Notice that, by well-known properties of the Fr\'echet subdifferential
(see, for instance, \cite{BorZhu05,Mord06,Schi07}), one has
$$
  \fsubdif\left(\mu+{\epsilon\over 2}\right)\|\cdot-p_\epsilon\|(p)
  \subseteq\left(\mu+{\epsilon\over 2}\right)\B^*,\quad\forall p\in P.
$$
Since the function $\dispo{F}(\cdot,x_\epsilon)+\left(\mu+{\epsilon\over 2}
\right)\|\cdot-p_\epsilon\|$, which is the sum of a l.s.c. and a
Lipschitz function, attains a local minimum at $p_\epsilon$, it is
possible to apply the Fr\'echet fuzzy sum rule in Definition \ref{def:fuzsumrul},
in force of hypothesis $(i)$. Accordingly, taken $\eta\in (0,\epsilon/2)$
in such a way that $\ball{p_\epsilon}{\eta}\times\{x_\epsilon\}\in W$,
one gets consequent $(p_i,x_\epsilon)$ and $p^*_i\in P^*$, $i=1,\, 2$,
such that
$$
   \|p_i-p_\epsilon\|<\eta,\qquad i=1,\, 2.
$$
$$
  p^*_1\in\fsubdif_p\dispo{F}(p_1,x_\epsilon),\qquad\quad
  p^*_2\in\fsubdif\left(\mu+{\epsilon\over 2}\right)\|\cdot-p_\epsilon\|
  (p_2),
$$
and
$$
   \|p^*_1+p^*_2\|<\eta.
$$
As it is $\eta<\epsilon/2$, one can deduce that $p^*_1\in(\mu+
\epsilon)\B^*$, and hence, since it is $(p_1,x_\epsilon)\in
\ball{p_\epsilon}{\eta}\times\{x_\epsilon\}\subseteq W$, one
obtains
$$
  \inf\{\|p^*\|:\ p^*\in\fsubdif_p{\dispo{F}}(p,x),\
  (p,x)\in W\}\le\mu+\epsilon.
$$
The thesis follows by arbitrariness of $\epsilon$.
\end{proof}

Now, for formulating the next technical lemma, some further notations
are needed. Given a set-valued mapping $F:P\times X\rightrightarrows
Y$ and $(p,x)\in P\times X$, let us set
\begin{eqnarray*}
  \sigma(p,x)=\lim_{\epsilon\to 0^+}\inf\{\|p^*\| &:&
  p^*\in\Coder F(\cdot,x)(p',y')(y^*),\ \|y^*\|=1,\\
  & & p'\in\ball{p}{\epsilon},\ y'\in Y: \|y'\|\le\dispo{F}
  (p',x)+\epsilon\},
\end{eqnarray*}
where $\Coder F(\cdot,x)(p',y'):Y^*\rightrightarrows P^*$ denotes
the Fr\'echet coderivative of the set-valued mapping $F(\cdot,
x):P\rightrightarrows Y$ (hence, the partial coderivative
of $F$ with respect to $p$), calculated at $(p',y')\in\grph
F(\cdot,x)$. Again, set
$$
   V_\eta=\inte[\ball{\bar p}{\eta}\times\ball{\bar x}{\eta}]
   \backslash F^{-1}(\nullv).
$$

\begin{remark}      \label{rem:uscopen}
Notice that, since $F:P\times X\rightrightarrows Y$ is closed
valued, one has
$$
  (P\times X)\backslash F^{-1}(\nullv)=\{(p,x)\in
   P\times X:\ \dispo{F}(p,x)>0\}.
$$
Therefore, whenever
$F$ is u.s.c. on a set $\inte[\ball{\bar p}{\delta_0}\times
\ball{\bar x}{\delta_0}]$, so that function $\dispo{F}:P\times
X\longrightarrow [0,+\infty]$ is
l.s.c. on the same set (remember Remark \ref{rem:uscHusc}),
each set $V_\eta$, with $\eta<\delta_0$ turns out to be open.
\end{remark}

\begin{lemma}     \label{lem:fsubdifcoder}
Let $F:W\rightrightarrows Y$ be a set-valued mapping between
Banach spaces and let $W\subseteq P\times X$ be an open set.
Suppose that:
\begin{enumerate}

\item[(i)] $(P,\|\cdot\|)$ and $(Y,\|\cdot\|)$ are Asplund;

\item[(ii)] there exists $\delta_0>0$ such that $F(\cdot,x)$ is
u.s.c. on $\ball{\bar p}{\delta_0}$, for each $x\in\ball{\bar x}{\delta_0}$;

\item[(iii)] it is $W\subseteq[\ball{\bar p}{\delta_0}\times
\ball{\bar x}{\delta_0}]\backslash F^{-1}(\nullv)$ and there exists
a constant $\sigma>0$ such that
$$
   \inf_{(p,x)\in W}\sigma(p,x)\ge\sigma.
$$
\end{enumerate}
Then, it holds
$$
   \inf\{\|p^*\|:\ p^*\in\fsubdif_p{\dispo{F}}(p,x),\
  (p,x)\in W\}\ge\sigma.
$$
\end{lemma}

\begin{proof}
The thesis follows at once from \cite[Lemma 5.5.4]{BorZhu05}.
Indeed, it suffices to replace the Fr\'echet subdifferential
and coderivative with their partial counterparts and to observe
that, in order to apply the Fr\'echet fuzzy sum rule, the
hypothesis about the Fr\'echet smoothness assumed in
\cite[Lemma 5.5.4]{BorZhu05}
can be replaced with the Asplund property of $P$ and $Y$.
Recall that the Cartesian product of Asplund spaces is still
Asplund (see \cite{Mord06}).
\end{proof}

By means of the above constructions, it is possible to
establish the following coderivative condition for the uniform
hemiregularity of the multifunction implicitly defined by
a problem $({\mathcal E})$, in a Banach space setting.

\begin{theorem}
Lat $F:P\times X\rightrightarrows Y$ be a set-valued mapping
between Banach spaces defining a parameterization ${\mathcal R}:
P\rightrightarrows X$ for the solution set $R$ of a problem
$({\mathcal E})$. Given $\bar p\in P$, let $\bar x\in
{\mathcal R}(\bar p)$. Suppose that:
\begin{enumerate}

\item[(i)] $(P,\|\cdot\|)$ and $(Y,\|\cdot\|)$ are Asplund;

\item[(ii)] there exists $\delta_0>0$ such that $F$ is u.s.c.
on $\ball{\bar p}{\delta_0}\times\ball{\bar x}{\delta_0}$;

\item[(iii)] $F(\bar p,\cdot):X\rightrightarrows Y$ is uniformly
Lipschitz l.s.c. at $(\bar x,\nullv)$;

\item[(iv)] it is
\begin{eqnarray}   \label{in:fcodnondegcond}
  \sigma=\lim_{\eta\to 0^+}\inf_{(p,x)\in V_\eta}\sigma(p,x)>0.
\end{eqnarray}
\end{enumerate}
Then, ${\mathcal R}$ is uniformly hemiregular at $(\bar p,\bar x)$
and the following estimate holds
\begin{eqnarray*}
   \usreg{{\mathcal R}}{(\bar p,\bar x)}\le
   {\uLlsc{F}{(\bar x,\nullv)}\over\sigma}.
\end{eqnarray*}
\end{theorem}

\begin{proof}
The proof clearly relies on the application of Theorem \ref{thm:imthreg}.
Let us check that all hypotheses of that theorem are actually fulfilled under
the current assumptions.

Hypothesis $(i)$ takes trivially place in a Banach space setting.
As to hypothesis $(ii)$, it suffices to recall Remark \ref{rem:uscHusc}.
It remains to show that condition $(\ref{in:fcodnondegcond})$
guarantees the validity of hypothesis $(iv)$. To this aim, let us
start with observing that, fixed an arbitrary $\zeta>0$, inequality
$(\ref{in:fcodnondegcond})$ implies that it is possible to find
$\eta\in (0,\delta_0/2)$ such that
$$
  \inf_{(p,x)\in V_{\eta}}\sigma(p,x)\ge\sigma-\zeta.
$$
Thus, by applying Lemma \ref{lem:fsubdifcoder} with $W=V_{\eta}$
(note that, under the current hypotheses, it is an open set
according to Remark \ref{rem:uscopen}), one finds
$$
   \inf\{\|p^*\|:\ p^*\in\fsubdif_p{\dispo{F}}(p,x),\
  (p,x)\in V_{\eta}\}\ge\sigma-\zeta.
$$
In turn, on account of Lemma \ref{lem:stslpfsubdif}, the last
inequality gives
$$
   \inf_{(p,x)\in V_{\eta}}\stslp{\dispo{F}}(p,x)\ge
   \sigma-\zeta.
$$
By recalling the definition of $\sostslp{\dispo{F}}(\bar p)$,
since for a proper $\epsilon>0$ it happens that
$\{(p,x)\in\ball{\bar p}{\epsilon}\times\ball{\bar x}{\epsilon}:\
0<\dispo{F}(p,x)<\epsilon\}\subseteq V_\eta$, one obtains
$$
   \sostslp{\dispo{F}}(\bar p)\ge\inf_{(p,x)\in V_{\eta}}
   \stslp{\dispo{F}}(p,x)\ge\sigma-\zeta.
$$
As $\zeta$ has been arbitrarily taken, the above inequality
completes the proof.
\end{proof}

\vskip1cm


\end{document}